\documentclass{amsart}
\usepackage{amssymb}
\usepackage{amsmath}
\usepackage{amsfonts}
\usepackage{hyperref}
\usepackage{yfonts}
\usepackage{yhmath}
\usepackage{framed}
\DeclareMathOperator{\sgn}{sgn}
\newtheorem{theorem}{Theorem}[section]

\newtheorem{prop}[theorem]{Proposition}

\newtheorem{definition}[theorem]{Definition}
\newtheorem{remark}[theorem]{Remark}

\newcommand{\R}{\mathbb{R}}

\newcommand{\T}{\mathcal T}
\newcommand{\gT}{\textfrak{T}}
\newcommand{\bx}{{\bf{\xi}}}
\newcommand{\bc}{\textbf{1}}
\newcommand{\p}{\partial}
\newcommand{\al}{\alpha}

\newcommand{\tm}{{\widetilde{\mu}}}

\begin{document}
\title[Unconditional well-posedness for the Kawahara equation]{Unconditional well-posedness for the Kawahara equation}

\author{Dan-Andrei Geba and Bai Lin}

\address{Department of Mathematics, University of Rochester, Rochester, NY 14627}
\email{dangeba@math.rochester.edu}
\address{Department of Mathematics, University of Rochester, Rochester, NY 14627}
\email{blin13@u.rochester.edu}
\date{}

\begin{abstract}
This article is concerned with the unconditional well-posedness for the Kawahara equation on the real line and shows that this holds true for initial data in $L^2(\R)$. This is achieved by applying an infinite iteration scheme of normal form reductions.
\end{abstract}


\maketitle

\section{Introduction}

In this paper, our focus is on studying the Cauchy problem associated to the Kawahara equation,
\begin{equation}
\begin{cases}
\p_tu+\beta \p_x^3u+\alpha \p_x^5u+\p_x(u^2)\,=\,0, \qquad u=u(t,x): \mathbb{R}\times \R \to \R,\\
u(0,x)\,=\,u_0(x),\qquad u_0\in H^s(\R),\\
\end{cases}
\label{main}
\end{equation}
where $\alpha\neq 0$ and $\beta$ are real parameters. A renormalization process allows one to work with $\alpha=-1$ and $\beta=0$, $1$, or $-1$. The Kawahara equation is known as a generalized KdV equation modeling arbitrarily small perturbations of solitary waves with tails that have oscillatory structure (e.g. Kawahara \cite{K72}, Hunter-Scheurle \cite{HS88}). 

A natural object to investigate for evolution equations is the well-posedness (WP) of the associated Cauchy problem. The local theory for \eqref{main} has received considerable interest, with contributions from Cui-Deng-Tao \cite{CDT06}, Wang-Cui-Deng \cite{WCD07}, Chen-Li-Miao-Wu \cite{CLMW09}, and Chen-Guo \cite{CG11}. It was completed by Kato \cite{K11}, who showed that \eqref{main} is locally well-posed (LWP) if $s\geq -2$ and ill-posed (IP) if $s<-2$. The latter is due to the associated flow map being discontinuous in the corresponding topologies. Following Kato's result, Okamoto \cite{O17} proved IP in the same regime by exhibiting a norm inflation phenomenon. In what concerns the global well-posedness (GWP) for \eqref{main}, the best result is also due to Kato \cite{K13-2} who showed that this is valid when $s\geq -38/21$. Previous GWP results are due to Yan-Li \cite{YL10} and Chen-Guo \cite{CG11}.

The natural solution space for our Cauchy problem is either
\begin{equation}
X=C([-T,T]; H^s(\R)), \quad T>0,\label{xt}
\end{equation}
or 
\begin{equation}
X=C(\R; H^s(\R)). 
\label{xr}
\end{equation}
However, for any of the WP results mentioned above, the solution was obtained to be unique only in $X\cap Y$, where $Y$ is an additional functional space, and this is due to the analytic toolbox used in the argument. Hence, an interesting topic to study about \eqref{main} is its \emph{unconditional} WP, which means WP with the uniqueness of solution derived directly in $X$.  Our main result provides an answer in this direction.

\begin{theorem}
The Cauchy problem \eqref{main} is unconditionally GWP in $L^2(\R)$.\label{main-th}
\end{theorem}

The unconditional WP for nonlinear dispersive equations has been extensively investigated for the past 25 years, with the emergence, in the last decade or so, of an influential method to prove it, based on normal form reductions (NFR). More recently, critical steps in streamlining this method were taken by Guo-Kwon-Oh \cite{GKO13}, Kwon-Oh-Yoon \cite{KOY18}, and Kishimoto \cite{K19}. In fact, we refer the reader to \cite{KOY18} and \cite{K19} for comprehensive expositions of these topics, which also include exhaustive lists of references.

The approach we take in proving Theorem \ref{main-th} follows closely the framework implemented by Kwon-Oh-Yoon, in which unconditional WP is reduced to the proof of a specific set of multilinear estimates. Actually, our argument mirrors to a great extent the one for Theorem 1.2 in \cite{KOY18}, which applies to the modified KdV equation
\begin{equation}
\p_tu- \p_x^3u\pm\p_x(u^3)\,=\,0.
\label{kdv}
\end{equation}
This is why, in what follows, we focus mainly on:
\begin{itemize}
\item presenting stand-alone arguments for the multilinear estimates related to the Kawahara equation;

\item providing the reader precise directions to the NFR methodology in \cite{KOY18} on how to implement these bounds to conclude Theorem \ref{main-th}.
\end{itemize}
Nevertheless, a review of all the key elements in the infinite iteration scheme of NFR is provided for the convenience of the reader.

Our work can be seen as yet another case study for the wide applicability and robustness of this technique and we believe the arguments in this paper are more transparent than in other related articles on highlighting all the features in the NFR procedure. This is due to the enhanced dispersive relation and the less intricate nonlinearity of the Kawahara equation.


\section{Fundamentals of the infinite iteration scheme of NFR}


\subsection{Basic notational conventions and terminology}
First, we agree to write $A\lesssim B$ when $A\leq CB$ and $C>0$ is a constant varying from line to line and depending on various fixed parameters. Moreover, we write $A\sim B$ to denote that both $A\lesssim B$ and $B\lesssim A$ are valid. We also use $A\ll B$ to denote that $A\leq \epsilon B$ for some small absolute constant $\epsilon>0$.

Secondly, for a function $h=h(\xi,\eta,\zeta)$ defined on $\R^3$, we convene to write
\[
\int\limits_{\xi=\eta+\zeta} h = \int_\R h(\xi,\eta, \xi-\eta)\,d\eta= \int_\R h(\xi,\xi-\zeta, \zeta)\,d\zeta.
\]
Next, we denote by\footnote{For clarity purposes, we are going to use $\mathcal{F}$ only on mathematical expressions of considerable width.}
\begin{equation*}
\mathcal{F}(w)(\xi)=\widehat{w}(\xi)=\int_{\R} e^{-ix\xi}\,w(x)\,dx
\end{equation*}
the Fourier transform of the function $w=w(x)$ defined on $\R$ and, for $f=f(t,x)$ defined on $\R^2$, we adopt the notational convention
\[
\widehat{f}(t,\xi)=\widehat{f(t,\cdot)}(\xi).
\]
Following this, we introduce the norms
\begin{equation*}
\|w\|_{\mathcal{F}L^\infty}=\|\widehat w\|_{L^\infty(\R)}, \qquad \|w\|_{H^s}=\|\langle \xi\rangle^s\widehat{w}(\xi)\|_{L^2_\xi(\R)},
\end{equation*}
with $\langle a\rangle=(1+a^2)^{1/2}$. For a set $D\subseteq\R^2$, $\bc_D$ stands for its characteristic function.

Finally, in connection to the Kawahara equation, we let
\[
S(t)=e^{t(\p_x^5-\beta\p_x^3)}
\]
denote its linear propagator and, for a function $u=u(t,x)$, we call
\begin{equation}
v(t)= S(-t)u(t)\quad (\text{i.e.}, \ \widehat{v}(t,\xi)=e^{-it(\xi^5+\beta\xi^3)}\widehat{u}(t,\xi)), \label{ir}
\end{equation}
its \emph{interaction representation}.


\subsection{Informal description of the iteration scheme}
The first step in the iteration scheme we want to implement is to write the Duhamel formulation for \eqref{main}, which is given by
\begin{equation*}
u(t)=S(t)u_0-\int_0^t S(t-\tau)(\p_x(u^2(\tau)))\,d\tau.
\end{equation*}
Using \eqref{ir}, we can easily rewrite this formulation as 
\begin{equation}
\widehat{v}(t,\xi)=\widehat{u_0}(\xi)-\int_0^t\int\limits_{\xi=\xi_1+\xi_2}i\xi \,e^{i\tau\Phi(\xi, \xi_1,\xi_2)}\,\widehat{v}(\tau,\xi_1)\,\widehat{v}(\tau,\xi_2)\,d\tau,
\label{v1}
\end{equation}
where
\begin{equation}
\aligned
\Phi(\xi, \xi_1,\xi_2)&=\, \xi_1^5+\xi_2^5-\xi^5+\beta(\xi_1^3+\xi_2^3-\xi^3)\\&=\, -\xi\, \xi_1\xi_2(5(\xi_1^2+\xi_1\xi_2+\xi_2^2)+3\beta)
\endaligned
\label{mod-f}
\end{equation}
and the last equality is justified by $\xi=\xi_1+\xi_2$. In what follows, $\Phi$ will be referred to as the \emph{modulation function}. 

Next, if we assume that $\Phi(\xi, \xi_1,\xi_2)\neq 0$ and we rely on 
\begin{equation*}
\p_{t}\left\{\frac{e^{it\Phi(\xi, \xi_1,\xi_2)}}{i\,\Phi(\xi, \xi_1,\xi_2)}\right\}= e^{it\Phi(\xi, \xi_1,\xi_2)}
\end{equation*}
and 
\begin{equation}
\p_{\tau}\widehat{v}(\tau,\xi)=-\int\limits_{\xi=\xi_1+\xi_2}i\xi\, e^{i\tau\Phi(\xi, \xi_1,\xi_2)}\,\widehat{v}(\tau,\xi_1)\,\widehat{v}(\tau,\xi_2),
\label{v1-2}
\end{equation}
then we can force an integration by parts with respect to $\tau$ in \eqref{v1} to derive
\begin{equation*}
\aligned
\widehat{v}(t,\xi)=\widehat{u_0}(\xi)-&\int_0^t\int\limits_{\xi=\xi_1+\xi_2}i\xi\, \p_{\tau}\left\{\frac{e^{i\tau\Phi(\xi, \xi_1,\xi_2)}}{i\,\Phi(\xi, \xi_1,\xi_2)}\right\}\,\widehat{v}(\tau,\xi_1)\,\widehat{v}(\tau,\xi_2)\,d\tau\\
=\widehat{u_0}(\xi)-&\int\limits_{\xi=\xi_1+\xi_2}\left\{i\xi \,\frac{e^{i\tau\Phi(\xi, \xi_1,\xi_2)}}{i\,\Phi(\xi, \xi_1,\xi_2)}\,\widehat{v}(\tau,\xi_1)\,\widehat{v}(\tau,\xi_2)\right\}\bigg |_{\tau=0}^{\tau=t}\\
+&\int_0^t\int\limits_{\xi=\xi_1+\xi_2}i\xi \,\frac{e^{i\tau\Phi(\xi, \xi_1,\xi_2)}}{i\,\Phi(\xi, \xi_1,\xi_2)}\,\p_{\tau}\left\{\widehat{v}(\tau,\xi_1)\,\widehat{v}(\tau,\xi_2)\right\}\,d\tau\\
=\widehat{u_0}(\xi)-&\int\limits_{\xi=\xi_1+\xi_2}\left\{i\xi \,\frac{e^{i\tau\Phi(\xi, \xi_1,\xi_2)}}{i\,\Phi(\xi, \xi_1,\xi_2)}\,\widehat{v}(\tau,\xi_1)\,\widehat{v}(\tau,\xi_2)\right\}\bigg |_{\tau=0}^{\tau=t}\\
+&2\int_0^t\int\limits_{\stackrel{\xi=\xi_1+\xi_2}{\xi_1=\eta_1+\eta_2}}\xi\xi_1\, \frac{e^{i\tau(\Phi(\xi, \xi_1,\xi_2)+\Phi(\xi_1, \eta_1,\eta_2))}}{i\,\Phi(\xi, \xi_1,\xi_2)}\,\widehat{v}(\tau,\eta_1)\,\widehat{v}(\tau,\eta_2)\,\widehat{v}(\tau,\xi_2)\,d\tau.
\endaligned
\end{equation*}

This line of reasoning can be continued with $$\Phi(\xi, \xi_1,\xi_2)\mapsto \Phi(\xi, \xi_1,\xi_2)+\Phi(\xi_1, \eta_1,\eta_2),$$ if one also knows that $\Phi(\xi, \xi_1,\xi_2)+\Phi(\xi_1, \eta_1,\eta_2)\neq 0$. Obviously, neither of the assumptions made so far can be guaranteed to hold true throughout the domains of integration. This is why one splits the original spatial integral in \eqref{v1} into two integrals corresponding to the regions 
\[
\{|\Phi(\xi, \xi_1,\xi_2)|=\,\text{large}\}\qquad \text{and} \qquad \{|\Phi(\xi, \xi_1,\xi_2)|=\,\text{small}\},
\] 
respectively, where the terms \emph{large} and \emph{small} are to be specified later. It is only for the former integral that we apply the integration by parts procedure, whereas the latter is dealt with as is. Similarly, for the resulting integral in the region $\{|\Phi(\xi, \xi_1,\xi_2)|=\,\text{large}\}$, we decompose into two integrals for which the domains are\footnote{It is worth noting that the meanings of \emph{large} and \emph{small} in the context of $\Phi(\xi, \xi_1,\xi_2)+\Phi(\xi_1, \eta_1,\eta_2)$ have to be readjusted from the ones used in connection to $\Phi(\xi, \xi_1,\xi_2)$.}
\[
\{|\Phi(\xi, \xi_1,\xi_2)|=\,\text{large},\ |\Phi(\xi, \xi_1,\xi_2)+\Phi(\xi_1, \eta_1,\eta_2)|=\,\text{large}\}
\]
and
\[\{|\Phi(\xi, \xi_1,\xi_2)|=\,\text{large}, \ |\Phi(\xi, \xi_1,\xi_2)+\Phi(\xi_1, \eta_1,\eta_2)|=\,\text{small}\}.
\] 
Yet again, an integration by parts with respect to $\tau$ is applied only for the former integral and the whole process continues along similar lines. 

The final goal of this iteration scheme is to reach the \emph{normal form equation} 
\begin{equation}
v(t)=u_0+\,\sum_{k=2}^\infty N_0^{(k)}(v(\tau))\,\bigg |_{\tau=0}^{\tau=t} \,+\, \int_0^t\,\sum_{k=1}^\infty N_1^{(k)}(v(\tau)) d\tau,
\label{nfe}
\end{equation}
where $N_0^{(k)}=N_0^{(k)}(w)$ and $N_1^{(k)}=N_1^{(k)}(w)$ are $k$-linear and $(k+1)$-linear expressions in $w$, respectively. For a fixed $k$, both $N_0^{(k)}$ and $N_1^{(k)}$ are the outcomes of performing $k$ iterations of the procedure outlined above. 


\subsection{Formal derivation of the normal form equation}\label{sect-iter}

In order to formally implement this approach to our Cauchy problem, we recall the notions of \emph{ordered tree} and \emph{index function} from Section 3 in \cite{KOY18}, which are adapted to the Kawahara equation (whose nonlinearity is bilinear). 

\begin{definition}
A finite, partially ordered set $(\T, \leq)$ is called a \emph{binary tree} if:

\begin{itemize}

\item there exists a maximal element $r\in\T$, which is also named a \emph{root node}.

\item for each $a\in\T\backslash \{r\}$, there exists a unique $b\in\T$ such that $a\leq b$ and $(a\leq c\leq b \Rightarrow c=a$ or $c=b)$. In this instance, $a$ is called a \emph{child} of $b$ and $b$ is called a \emph{parent} of $a$.

\item each parent $b\in\T$ has exactly two children, labeled in the planar graphical representation of $\T$ (from left to right) as $b_1$ and $b_2$.

\end{itemize}
For obvious reasons, a parent in $\T$ is also named a \emph{non-terminal node} of $\T$. The set of all non-terminal nodes in $\T$ is denoted by $\T^0$, while $T^\infty=T\,\backslash\, T^0$ stands for set of all \emph{terminal nodes} in $\T$.  
\end{definition}

\begin{remark}
For a binary tree $\T$, it is straightforward to see that $|\T|=2k+1$, $|\T^0|=k$, and $|\T^\infty|=k+1$, where $k\geq 1$ is a specific integer. Moreover, one has $|\gT(k)|=k!$, where $\gT(k)$ denotes the set of all binary trees with $k$ non-terminal nodes.
\end{remark}

\begin{definition}
A sequence $(\T_k)_{1\leq k\leq K}$ of binary trees is called a \emph{chronicle of $K$ generations} if:

\begin{itemize}

\item $\T_k\in\gT(k) $ for all $1\leq k\leq K$.

\item for each $1\leq k\leq K-1$, $\T_{k+1}$ is obtained from $\T_k$ by changing one of its terminal nodes into a non-terminal one (with two children).

\end{itemize}
The binary tree $\T_K$ in a chronicle of $K$ generations is known as an \emph{ordered tree of the $K$th generation}.  
\end{definition}

\begin{definition}
For an ordered tree $\T\in\gT(k)$, we call $\bx: \T\to\R$ (identified with $(\bx_a)_{a\in\T}$) an \emph{index function} if
\begin{equation*}
\xi_b=\xi_{b_1}+\xi_{b_2}, \quad (\forall)\, b\in\T^0,
\end{equation*}
where $b_1$ and $b_2$ are the children of $b$. The set of all such index functions is denoted by $\Xi(\T)$. 
\end{definition}

To streamline how we write various terms appearing in the iteration scheme, we rely both on ordered trees and index functions, as well as superscripts in connection to \emph{generations of frequencies}. Thus, for an ordered tree $\T\in\gT(K)$ with its chronicle of generations $(\T_k)_{1\leq k\leq K}$ and associated index function $\bx\in\Xi(\T)$, we let
\begin{equation*}
(\xi^{(1)}, \xi_1^{(1)}, \xi_2^{(1)})=(\xi_r,\xi_{r_1},\xi_{r_2})
\end{equation*}
be the first generation of frequencies, where $r$ is the root node of $\T_K$ and $r_1$ and $r_2$ are its children. Similarly, for $k\geq 2$, we denote
\begin{equation*}
(\xi^{(k)}, \xi_1^{(k)}, \xi_2^{(k)})=(\xi_b,\xi_{b_1},\xi_{b_2})
\end{equation*}
to be the $k$th generation of frequencies, where $b$ is the terminal node of $\T_{k-1}$ changed into a non-terminal one for $\T_k$ and $b_1$ and $b_2$ are its children in $\T_k$. Accordingly, we work with the following notation for the modulation function introduced for the $k$th generation of frequencies,
\begin{equation*}
\mu_k=\Phi(\xi^{(k)}, \xi_1^{(k)}, \xi_2^{(k)})=  (\xi^{(k)}_1)^5+(\xi^{(k)}_2)^5-(\xi^{(k)})^5+\beta((\xi^{(k)}_1)^3+(\xi^{(k)}_2)^3-(\xi^{(k)})^3),
\end{equation*}
and we also write
\begin{equation*}
\tilde{\mu}_k=\sum_{j=1}^k \mu_j.
\end{equation*}

Now, we have all the prerequisites to formally describe the iteration scheme. We start with \eqref{v1} written as
\begin{equation*}
\widehat{v}(t,\xi)=\widehat{u_0}(\xi)+\int_0^t\,\widehat{N(v)}(\tau,\xi)\,d\tau,
\end{equation*}
with
\begin{equation*}
\widehat{N(v)}(\tau,\xi)=-\int\limits_{\xi=\xi_1+\xi_2}i\xi \,e^{i\tau\Phi(\xi, \xi_1,\xi_2)}\,\widehat{v}(\tau,\xi_1)\,\widehat{v}(\tau,\xi_2).
\end{equation*}
Next, we fix $N>1$ to be a large, dyadic parameter and write
\begin{equation*}
N(v)=N^{(1)}_1(v)+N^{(1)}_2(v),
\end{equation*}
where\footnote{$N^{(1)}_2$ is obviously defined in complementary fashion using $N$ and $N^{(1)}_1$.}
\begin{equation*}
\widehat{N^{(1)}_1(v)}(\tau,\xi)=-\int\limits_{\xi=\xi_1+\xi_2}\bc_{C^c_0}\,i\xi \,e^{i\tau\Phi(\xi, \xi_1,\xi_2)}\,\widehat{v}(\tau,\xi_1)\,\widehat{v}(\tau,\xi_2)
\end{equation*}
and $C_0=\{|\Phi(\xi,\xi_1,\xi_2)|>N\}$. Hence,
\begin{equation*}
\p_tv=N(v)=N^{(1)}_1(v)+N^{(1)}_2(v).
\end{equation*}
Moreover, in the language of ordered trees and index functions\footnote{Onward, for ease of notation, we write $\widehat{v}_{\xi_a}$ for $\widehat{v}(\tau, \xi_a)$ in the integral terms.}, we have $C_0=\{|\tm_1|=|\mu_1|>N\}$,
\begin{align}
\widehat{N^{(1)}_1(v)}(\tau,\xi)&=-\sum_{\T_1\in \gT(1)}\,\int\limits_{\stackrel{\bx\in\Xi(\T_1)}{\bx_r=\xi}} \bc_{C^c_0}\,i\xi^{(1)}\,e^{i\tau\tm_1}\prod_{a\in \T_1^\infty}\,\widehat{v}_{\xi_a},\label{n11}\\
\widehat{N^{(1)}_2(v)}(\tau,\xi)&=-\sum_{\T_1\in \gT(1)}\,\int\limits_{\stackrel{\bx\in\Xi(\T_1)}{\bx_r=\xi}} \bc_{C_0}\,i\xi^{(1)}\,e^{i\tau\tm_1}\prod_{a\in \T_1^\infty}\,\widehat{v}_{\xi_a}.
\end{align}

Since $\mu_1\neq 0$ on $C_0$, we infer based on \eqref{v1-2} that
\begin{equation}
\aligned
\widehat{N^{(1)}_2(v)}(\tau,\xi)=\ &\p_{\tau}\left\{-\sum_{\T_1\in \gT(1)}\,\int\limits_{\stackrel{\bx\in\Xi(\T_1)}{\bx_r=\xi}} \bc_{C_0}\,i\xi^{(1)}\,\frac{e^{i\tau\tm_1}}{i\tm_1}\prod_{a\in \T_1^\infty}\,\widehat{v}_{\xi_a}\right\}\\&\qquad+\sum_{\T_1\in \gT(1)}\,\int\limits_{\stackrel{\bx\in\Xi(\T_1)}{\bx_r=\xi}} \bc_{C_0}\,i\xi^{(1)}\,\frac{e^{i\tau\tm_1}}{i\tm_1}\p_{\tau}\left\{\prod_{a\in \T_1^\infty}\,\widehat{v}_{\xi_a}\right\}\\
=\ &\p_{\tau}\left\{-\sum_{\T_1\in \gT(1)}\,\int\limits_{\stackrel{\bx\in\Xi(\T_1)}{\bx_r=\xi}} \bc_{C_0}\,i\xi^{(1)}\,\frac{e^{i\tau\tm_1}}{i\tm_1}\prod_{a\in \T_1^\infty}\,\widehat{v}_{\xi_a}\right\}\\&\qquad-\sum_{\T_2\in \gT(2)}\,\int\limits_{\stackrel{\bx\in\Xi(\T_2)}{\bx_r=\xi}} \bc_{C_0}\,i\xi^{(1)}\,i\xi^{(2)}\,\frac{e^{i\tau\tm_2}}{i\tm_1}\prod_{a\in \T_2^\infty}\,\widehat{v}_{\xi_a}.
\endaligned
\label{n1-2}
\end{equation}
By introducing $N_0^{(2)}=N_0^{(2)}(v)$ and $N^{(2)}=N^{(2)}(v)$ according to
\begin{align}
\widehat{N^{(2)}_0(v)}(\tau,\xi)&=-\sum_{\T_1\in \gT(1)}\,\int\limits_{\stackrel{\bx\in\Xi(\T_1)}{\bx_r=\xi}} \bc_{C_0}\,i\xi^{(1)}\,\frac{e^{i\tau\tm_1}}{i\tm_1}\prod_{a\in \T_1^\infty}\,\widehat{v}_{\xi_a},\label{n02}\\\widehat{N^{(2)}(v)}(\tau,\xi)&=-\sum_{\T_2\in \gT(2)}\,\int\limits_{\stackrel{\bx\in\Xi(\T_2)}{\bx_r=\xi}} \bc_{C_0}\,i\xi^{(1)}\,i\xi^{(2)}\,\frac{e^{i\tau\tm_2}}{i\tm_1}\prod_{a\in \T_2^\infty}\,\widehat{v}_{\xi_a},
\end{align}
we have that
\begin{equation*}
N^{(1)}_2(v)= \p_tN^{(2)}_0(v)+N^{(2)}(v).
\end{equation*}
Following this, we let $C_1=\{|\tm_2|\leq 5^3|\tm_1|^{1-\delta}\}$, where $\delta>0$ is a small, fixed parameter, and write the decomposition 
\begin{equation*}
N^{(2)}(v)=N_1^{(2)}(v)+N_2^{(2)}(v)
\end{equation*}
with
 \begin{align}
\widehat{N^{(2)}_1(v)}(\tau,\xi)&=-\sum_{\T_2\in \gT(2)}\,\int\limits_{\stackrel{\bx\in\Xi(\T_2)}{\bx_r=\xi}} \bc_{C_0\cap C_1}\,i\xi^{(1)}\,i\xi^{(2)}\,\frac{e^{i\tau\tm_2}}{i\tm_1}\prod_{a\in \T_2^\infty}\,\widehat{v}_{\xi_a},\label{n12}\\
\widehat{N_2^{(2)}(v)}(\tau,\xi)&=-\sum_{\T_2\in \gT(2)}\,\int\limits_{\stackrel{\bx\in\Xi(\T_2)}{\bx_r=\xi}} \bc_{C_0\cap C_1^c}\,i\xi^{(1)}\,i\xi^{(2)}\,\frac{e^{i\tau\tm_2}}{i\tm_1}\prod_{a\in \T_2^\infty}\,\widehat{v}_{\xi_a}.
\end{align}
Thus, we deduce that
\begin{equation*}
\p_tv= \p_tN^{(2)}_0(v)+\left(N^{(1)}_1(v)+N^{(2)}_1(v)\right)+N^{(2)}_2(v).
\end{equation*}

Next, we notice that 
\begin{equation*}
|\tm_2|> 5^3|\tm_1|^{1-\delta}>5^3N^{1-\delta}
\end{equation*}
holds true on $C_0\cap C_1^c$ and, hence, we can argue like in the derivation of \eqref{n1-2} to obtain
\begin{equation*}
N^{(2)}_2(v)=\p_tN_0^{(3)}(v)+N^{(3)}(v)
\end{equation*}
with
\begin{equation*}
\aligned
\widehat{N^{(3)}_0(v)}(\tau,\xi)=\ &-\sum_{\T_2\in \gT(2)}\,\int\limits_{\stackrel{\bx\in\Xi(\T_2)}{\bx_r=\xi}} \bc_{C_0\cap C_1^c}\,i\xi^{(1)}\,i\xi^{(2)}\,\frac{e^{i\tau\tm_2}}{i\tm_1i\tm_2}\prod_{a\in \T_2^\infty}\,\widehat{v}_{\xi_a},\\
\widehat{N^{(3)}(v)}(\tau,\xi)=\ &-\sum_{\T_3\in \gT(3)}\,\int\limits_{\stackrel{\bx\in\Xi(\T_3)}{\bx_r=\xi}} \bc_{C_0\cap C_1^c}\,i\xi^{(1)}\,i\xi^{(2)}\,i\xi^{(3)}\,\frac{e^{i\tau\tm_3}}{i\tm_1i\tm_2}\prod_{a\in \T_3^\infty}\,\widehat{v}_{\xi_a}.
\endaligned
\end{equation*}
At this point, we introduce $C_2=\{|\tm_3|\leq 7^3\max \{|\tm_1|, |\tm_2|\}^{1-\delta}\}$ and break up $N^{(3)}$ into
\begin{equation*}
N^{(3)}(v)=N_1^{(3)}(v)+N_2^{(3)}(v),
\end{equation*}
with
\begin{equation*}
\aligned
\widehat{N^{(3)}_1(v)}(\tau,\xi)&=-\sum_{\T_3\in \gT(3)}\,\int\limits_{\stackrel{\bx\in\Xi(\T_3)}{\bx_r=\xi}} \bc_{C_0\cap C_1^c\cap C_2}\,i\xi^{(1)}\,i\xi^{(2)}\,i\xi^{(3)}\,\frac{e^{i\tau\tm_3}}{i\tm_1i\tm_2}\prod_{a\in \T_3^\infty}\,\widehat{v}_{\xi_a},\\
\widehat{N_2^{(3)}(v)}(\tau,\xi)&=-\sum_{\T_3\in \gT(3)}\,\int\limits_{\stackrel{\bx\in\Xi(\T_3)}{\bx_r=\xi}} \bc_{C_0\cap C_1^c\cap C_2^c}\,i\xi^{(1)}\,i\xi^{(2)}\,i\xi^{(3)}\,\frac{e^{i\tau\tm_3}}{i\tm_1i\tm_2}\prod_{a\in \T_3^\infty}\,\widehat{v}_{\xi_a}.
\endaligned
\end{equation*} 
Hence, after this third step, we arrive at
\begin{equation*}
\p_tv=\p_t\left\{N^{(2)}_0(v)+N^{(3)}_0(v)\right\}+\left(N^{(1)}_1(v)+N^{(2)}_1(v)+N^{(3)}_1(v)\right)+N^{(3)}_2(v).
\end{equation*}

Continuing in the same vein, the $k$th step ($k\geq 3$) brings about $C_{k-1}=\{|\tm_k|\leq (2k+1)^3\max \{|\tm_1|, |\tm_{k-1}|\}^{1-\delta}\}$,
\begin{equation}
\aligned
\widehat{N^{(k)}_0(v)}&(\tau,\xi)\\&=-\sum_{\T_{k-1}\in \gT(k-1)}\,\int\limits_{\stackrel{\bx\in\Xi(\T_{k-1})}{\bx_r=\xi}} \bc_{C_0\cap\, \bigcap_{j=1}^{k-2}C_j^c}\,\prod_{j=1}^{k-1}i\xi^{(j)}\,\frac{e^{i\tau\tm_{k-1}}}{\prod_{j=1}^{k-1}i\tm_j}\prod_{a\in \T_{k-1}^\infty}\,\widehat{v}_{\xi_a},\label{n0k}
\endaligned
\end{equation}
\begin{align}
\widehat{N^{(k)}(v)}(\tau,\xi)&=-\sum_{\T_k\in \gT(k)}\,\int\limits_{\stackrel{\bx\in\Xi(\T_k)}{\bx_r=\xi}} \bc_{C_0\cap\, \bigcap_{j=1}^{k-2}C_j^c}\,\prod_{j=1}^{k}i\xi^{(j)}\,\frac{e^{i\tau\tm_k}}{\prod_{j=1}^{k-1}i\tm_j}\prod_{a\in \T_k^\infty}\,\widehat{v}_{\xi_a},
\label{nk}
\end{align}
\begin{equation}
\aligned
\widehat{N^{(k)}_1(v)}&(\tau,\xi)\\&=-\sum_{\T_k\in \gT(k)}\,\int\limits_{\stackrel{\bx\in\Xi(\T_k)}{\bx_r=\xi}} \bc_{C_0\cap\, \bigcap_{j=1}^{k-2}C_j^c \cap C_{k-1}}\,\prod_{j=1}^{k}i\xi^{(j)}\,\frac{e^{i\tau\tm_k}}{\prod_{j=1}^{k-1}i\tm_j}\prod_{a\in \T_k^\infty}\,\widehat{v}_{\xi_a},\label{n1k}
\endaligned
\end{equation}
\begin{align}
\widehat{N_2^{(k)}(v)}(\tau,\xi)&=-\sum_{\T_k\in \gT(k)}\,\int\limits_{\stackrel{\bx\in\Xi(\T_k)}{\bx_r=\xi}} \bc_{C_0\cap\, \bigcap_{j=1}^{k-1}C_j^c}\,\prod_{j=1}^{k}i\xi^{(j)}\,\frac{e^{i\tau\tm_k}}{\prod_{j=1}^{k-1}i\tm_j}\prod_{a\in \T_k^\infty}\,\widehat{v}_{\xi_a},
\label{n2k}
\end{align}
for which
\begin{equation*}
N^{(k-1)}_2(v)=\p_tN_0^{(k)}(v)+N^{(k)}(v)=\p_tN_0^{(k)}(v)+N_1^{(k)}(v)+N_2^{(k)}(v)
\end{equation*}
and, consequently,
\begin{equation}
\p_tv=\p_t\left\{\sum_{j=2}^k N^{(j)}_0(v)\right\}+\sum_{j=1}^k N^{(j)}_1(v)+N^{(k)}_2(v).
\label{nfe-k}
\end{equation}
Thus, one formally obtains \eqref{nfe} by performing infinite iterations of this scheme. 


\subsection{Summary of key items complementary to the iteration scheme in the NFR methodology}\label{just}
Following the formal derivation of the normal form equation, the analysis shifts now to the two main tasks left to validate in order to claim Theorem \ref{main-th}. The first one consists in proving that the Cauchy problem associated to \eqref{nfe} is LWP in $H^s$ for $s\geq 0$, which is enough for our purposes given Kato's conditional GWP result when $s\geq-38/21$. In this direction, we develop localized bilinear estimates in $H^s$ and highlight arguments similar to the ones in Sections 3.2-3.4 of \cite{KOY18} on how successive applications of these bounds yield the desired goal.

The second task, which is more involved, has to do with justifying certain steps in the iteration scheme which led us from the Kawahara equation to the normal form equation. On one hand, we need to show that $L^2$-solutions of the former satisfy each of the intermediate equations \eqref{nfe-k}. This means that we can rigorously apply Leibniz's rule of differentiation with respect to the $\tau$ variable and commute the differentiation with respect to $\tau$ with the integration in the spatial frequencies. On the other hand, we have to prove that \eqref{nfe} is a limiting value for \eqref{nfe-k} in the sense that $\widehat{N_2^{(k)}(v)}\to 0$ pointwise in $(\tau,\xi)$ as $k\to \infty$. For this purpose, we develop localized bilinear modulation estimates involving both $\mathcal{F}L^\infty$ and $H^s$ and refer the reader to arguments in Sections 4.1-4.2 of \cite{KOY18} which rely on these types of bounds to infer the mathematical facts described above. 


\section{Localized modulation estimates in $H^s$}\label{mod-hs}
In this section, we prepare a number of bilinear estimates to be used in the analysis of the normal form equation. To this end, we introduce the bilinear operators $N^{\al}_{\leq M}=N^{\al}_{\leq M}(v_1,v_2)$ and $I^{\al}_{>M}=I^{\al}_{>M}(v_1,v_2)$ by\footnote{For similar definitions, see Section 2.2 of \cite{KOY18}.}
\begin{equation}
\widehat{N^{\al}_{\leq M}(v_1,v_2)}(t,\xi)= -\int\limits_{\stackrel{\xi=\xi_1+\xi_2}{|\Phi(\xi, \xi_1,\xi_2)-\al |\leq M}}i\xi\, e^{it\Phi(\xi, \xi_1,\xi_2)}\,\widehat{v_1}(\xi_1)\,\widehat{v_2}(\xi_2)
\label{nam}
\end{equation}
and
\begin{equation}
\widehat{I^{\al}_{> M}(v_1,v_2)}(t,\xi)= -\int\limits_{\stackrel{\xi=\xi_1+\xi_2}{|\Phi(\xi, \xi_1,\xi_2)-\al |>  M}}i\xi\, \frac{e^{it\Phi(\xi, \xi_1,\xi_2)}}{\Phi(\xi, \xi_1,\xi_2)-\al }\,\widehat{v_1}(\xi_1)\,\widehat{v_2}(\xi_2),
\label{iam}
\end{equation}
with related definitions for $N^{\al}_{M}=N^{\al}_{M}(v_1,v_2)$ and $I^{\al}_{M}=I^{\al}_{M}(v_1,v_2)$ in which the restriction in the domain of the integral is now $M<|\Phi(\xi, \xi_1,\xi_2)-\al |\leq 2M$. Above, $M\geq 1$ and $\al$ are real parameters and $\Phi$ is given by \eqref{mod-f}. 

\begin{prop}
If $s\geq 0$, then the following estimates hold true:
\begin{equation}
\|N^{\al}_{\leq  M}(v_1,v_2)(t)\|_{H^s}\lesssim M^{1/2}\|v_1\|_{H^s}\|v_2\|_{H^s},\label{nm}
\end{equation}
\begin{equation}
\aligned
\|N^{\al}_{\leq M}(v_1,v_1)(t)-N^{\al}_{\leq  M}(v_2,&v_2)(t) \|_{H^s}\\
&\lesssim\, M^{1/2}\|v_1-v_2\|_{H^s}(\|v_1\|_{H^s}+\|v_2\|_{H^s}),
\endaligned
\label{nmd}
\end{equation}
\begin{equation}
\|I^{\al}_{> M}(v_1,v_2)(t)\|_{H^s}\lesssim M^{-1/2}\|v_1\|_{H^s}\|v_2\|_{H^s},
\label{im}
\end{equation}
\begin{equation}
\aligned
\|I^{\al}_{> M}(v_1,v_1)(t)-I^{\al}_{> M}(v_2,&v_2)(t) \|_{H^s}\\
&\lesssim\, M^{-1/2}\|v_1-v_2\|_{H^s}(\|v_1\|_{H^s}+\|v_2\|_{H^s}).
\label{imd}
\endaligned
\end{equation}
$N^{\al}_{M}$ and $I^{\al}_{M}$ satisfy bounds identical to the ones for $N^{\al}_{\leq M}$ and $I^{\al}_{>M}$, respectively.
\label{prop-nm}
\end{prop}

\begin{proof}
We start by showing that \eqref{nm} is valid and we use duality to argue that it is sufficient to prove that
\begin{equation}
\aligned
\int\limits_\R\int\limits_{\xi=\xi_1+\xi_2}\bc_{\{|\Phi(\xi, \xi_1,\xi_2)-\al |\leq M\}}\frac{|\xi|\langle \xi\rangle^s}{\langle \xi_1\rangle^s\langle \xi_2\rangle^s}\, &w_1(\xi_1)\,w_2(\xi_2)\,w(\xi)\,d\xi\\
&\lesssim M^{1/2}\|w_1\|_{L^2} \|w_2\|_{L^2} \|w\|_{L^2},
\endaligned
\label{nmw}
\end{equation}
where $w_1$, $w_2$, and $w\in L^2$ are all non-negative functions. Next, on the account of the triangle inequality and $s\geq 0$, we have that $\langle \xi\rangle^s\leq\langle \xi_1\rangle^s\langle \xi_2\rangle^s$ and, thus, it is enough to prove the claim for $s=0$. As in the corresponding result for \eqref{kdv} (i.e., Lemma 2.6 in \cite{KOY18}), we perform an analysis which takes into account the sizes of various spatial frequencies. We can assume $|\xi_1|\geq |\xi_2|$ since \eqref{nmw} exhibits symmetry in the indices $1$ and $2$. Accordingly, we split the analysis in the following complementary scenarios: 
\[|\xi|\lesssim 1, \qquad \max\{|\xi_2|,1\}\ll |\xi|,\qquad \text{and} \qquad 1\ll |\xi|\lesssim |\xi_1|\sim |\xi_2|.
\]

\underline{Case: $|\xi|\lesssim 1$.}\ Here, we use the Cauchy-Schwarz inequality twice to deduce
\begin{equation*}
\aligned
\text{LHS of \eqref{nmw}}\ &\lesssim \bigg\|\int\limits_{\xi=\xi_1+\xi_2}w_1(\xi_1)\,w_2(\xi_2)\bigg\|_{L^2_{\{|\xi|\lesssim 1\}}}\|w\|_{L^2}\\  
&\lesssim \bigg\|\int\limits_{\xi=\xi_1+\xi_2}w_1(\xi_1)\,w_2(\xi_2)\bigg\|_{L^\infty_{\{|\xi|\lesssim 1\}}}\|w\|_{L^2}\lesssim \|w_1\|_{L^2} \|w_2\|_{L^2} \|w\|_{L^2},
\endaligned
\end{equation*}
which proves the claim.

\underline{Case: $\max\{|\xi_2|,1\}\ll |\xi|$.}\ In this instance, another double application of the Cauchy-Schwarz inequality yields
\begin{equation*}
\aligned
\text{LHS of \eqref{nmw}}\ &\lesssim \bigg\|\, |\xi|\int\limits_{\xi=\xi_1+\xi_2}\bc_{\{|\Phi(\xi, \xi_1,\xi_2)-\al |\leq M\}}w_1(\xi_1)\,w_2(\xi_2)\bigg\|_{L^2}\|w\|_{L^2}\\  
&\lesssim \sup_{\xi\in \R}\bigg(\xi^2\int\limits_{\xi=\xi_1+\xi_2}\bc_{\{|\Phi(\xi, \xi_1,\xi_2)-\al |\leq M\}}\bigg)^{1/2}\|w_1\|_{L^2} \|w_2\|_{L^2} \|w\|_{L^2}.
\endaligned
\end{equation*}
Hence, it suffices to show that
\begin{equation*}
\sup_{|\xi|\gg 1}\xi^2\int\limits_{\stackrel{\xi=\xi_1+\xi_2}{\max\{|\xi_2|,1\}\ll |\xi|}}\bc_{\{|\Phi(\xi, \xi_1,\xi_2)-\al |\leq M\}}\lesssim M.
\end{equation*}
For this purpose, we fix $\xi$, write $G(\xi_1)=\Phi(\xi_1,\xi-\xi_1, \xi)$, and compute
\begin{equation}
G'(\xi_1)=(\xi_1^2-(\xi-\xi_1)^2)(5(\xi_1^2+(\xi-\xi_1)^2)+3\beta).
\label{g'}
\end{equation}
Since $|\xi|=|\xi_1+\xi_2|\gg |\xi_2|$, it follows that $|\xi_1|\sim|\xi|\gg|\xi-\xi_1|$ and, consequently, $G'(\xi_1)\sim \xi^4$. On the other hand, on the domain of integration, we have $|G(\xi_1)-\al|\leq M$, which implies that the range of $G$ has size $\lesssim M$. In the context of the mean value theorem, the last two facts tell us that the domain in which $\xi_1$ varies has size $\lesssim M/\xi^4$ and we conclude that  
\begin{equation*}
\sup_{|\xi|\gg 1} \xi^2\int\limits_{\stackrel{\xi=\xi_1+\xi_2}{\max\{|\xi_2|,1\}\ll |\xi|}}\bc_{\{|\Phi(\xi, \xi_1,\xi_2)-\al |\leq M\}}\lesssim \sup_{|\xi|\gg 1} \frac{M}{\xi^2}\lesssim M.
\end{equation*}

\underline{Case: $1\ll |\xi|\lesssim |\xi_1|\sim |\xi_2|$.}\ For this scenario, we first argue that the triangle inequality leads to 
\begin{equation*}
\max\{|\xi+\xi_1|, |\xi+\xi_2|\}\sim |\xi+\xi_1|+ |\xi+\xi_2|\geq 3|\xi|
\end{equation*}
and we choose to work onward with the assumption $\max\{|\xi+\xi_1|, |\xi+\xi_2|\}=|\xi+\xi_1|$. Yet again, we rely twice on the Cauchy-Schwarz inequality to infer that 
\begin{equation*}
\aligned
&\text{LHS of \eqref{nmw}}\ \lesssim \left\|\ \int\limits_{\stackrel{|\xi_1|\sim|\xi_2|\gtrsim |\xi|\gg 1}{|\xi+\xi_1|\geq |\xi+\xi_2|}}\bc_{\{|\Phi(\xi, \xi_1,\xi_2)-\al |\leq M\}}|\xi|\,w_1(\xi_1)\,w(\xi)\,d\xi_1\right\|_{L^2_{\xi_2}}\|w_2\|_{L^2}\\  
&\qquad\lesssim \sup_{|\xi_2|\gg 1}\left(\ \int\limits_{\stackrel{|\xi_1|\sim|\xi_2|\gtrsim |\xi|\gg 1}{|\xi+\xi_1|\geq |\xi+\xi_2|}}\bc_{\{|\Phi(\xi, \xi_1,\xi_2)-\al |\leq M\}}\,|\xi|^2\,d\xi_1\right)^{1/2}\|w_1\|_{L^2} \|w_2\|_{L^2} \|w\|_{L^2}.
\endaligned
\end{equation*}
Thus, what we need to prove is
\begin{equation*}
\sup_{|\xi_2|\gg 1} \int\limits_{\stackrel{|\xi_1|\sim|\xi_2|\gtrsim |\xi|\gg 1}{|\xi+\xi_1|\geq |\xi+\xi_2|}}\bc_{\{|\Phi(\xi, \xi_1,\xi_2)-\al |\leq M\}}\,|\xi|^2\,d\xi_1\lesssim M.
\end{equation*}
To this end, we proceed in similar fashion to the second case. We fix $\xi_2$, write $H(\xi_1)=\Phi(\xi_1,\xi_2, \xi_1+\xi_2)$, and compute
\begin{equation}
H'(\xi_1)=-\xi_2(\xi+\xi_1)(5(\xi_1^2+\xi^2)+3\beta).
\label{h'}
\end{equation}
Since $1\ll |\xi|\lesssim |\xi_1|\sim |\xi_2|$ and $|\xi+\xi_1| \gtrsim |\xi|$, it follows that $|H'(\xi_1)|\gtrsim|\xi_2|^3$ and, arguing as before, we derive that the set in which $\xi_1$ takes values has size $\lesssim M/|\xi_2|^3$. Hence, we obtain
\begin{equation*}
\sup_{|\xi_2|\gg 1} \int\limits_{\stackrel{|\xi_1|\sim|\xi_2|\gtrsim |\xi|\gg 1}{|\xi+\xi_1|\geq |\xi+\xi_2|}}\bc_{\{|\Phi(\xi, \xi_1,\xi_2)-\al |\leq M\}}\,|\xi|^2\,d\xi_1\lesssim \sup_{|\xi_2|\gg 1} \frac{M}{|\xi_2|}\lesssim M
\end{equation*}
and the proof of \eqref{nm} is complete.

Next, we observe that \eqref{nm} implies \eqref{nmd}, since a symmetric bilinear operator $T$ satisfies
\begin{equation*}
T(v_1,v_1)-T(v_2,v_2)= T(v_1,v_1-v_2)+T(v_2,v_1-v_2).
\end{equation*}
Moreover, the estimates for $N^{\al}_{M}$ follow at once from \eqref{nm} and \eqref{nmd} since
\begin{equation*}
N^{\al}_{M}=N^{\al}_{\leq 2M}-N^{\al}_{\leq M}.
\end{equation*}

For $I^{\al}_{M}$, we notice first that \eqref{nmw} is equivalent to
\begin{equation*}
\aligned
\bigg\|\ \int\limits_{\xi=\xi_1+\xi_2}\bc_{\{|\Phi(\xi, \xi_1,\xi_2)-\al |\leq M\}}|\xi|\,\langle\xi\rangle^s |\widehat{v_1}(\xi_1)|\,|\widehat{v_2}(\xi_2)|\bigg\|_{L^2}\lesssim M^{1/2} \|v_1\|_{H^s} \|v_2\|_{H^s}.
\endaligned
\end{equation*}
Therefore, we infer directly from the definition of $I^{\al}_{M}$ that 
\begin{equation*}
\aligned
\|I^{\al}_{M}(v_1,v_2)(t)\|_{H^s}&\lesssim \frac{1}{M}\bigg\|\ \int\limits_{\xi=\xi_1+\xi_2}\bc_{\{|\Phi(\xi, \xi_1,\xi_2)-\al |\sim M\}}|\xi|\,\langle\xi\rangle^s |\widehat{v_1}(\xi_1)|\,|\widehat{v_2}(\xi_2)|\bigg\|_{L^2}\\&\lesssim M^{-1/2} \|v_1\|_{H^s} \|v_2\|_{H^s}. 
\endaligned
\end{equation*}
This estimate implies \eqref{im} since
\begin{equation*}
\aligned
\|I^{\al}_{>M}(v_1,v_2)(t)\|_{H^s}&\lesssim \sum_{\stackrel{N\geq M}{N\,\text{dyadic}}}\|I^{\al}_{N}(v_1,v_2)(t)\|_{H^s}\\
&\lesssim \sum_{\stackrel{N\geq M}{N\,\text{dyadic}}}N^{-1/2}\|v_1\|_{H^s} \|v_2\|_{H^s}\lesssim M^{-1/2} \|v_1\|_{H^s} \|v_2\|_{H^s}. 
\endaligned
\end{equation*}
The arguments for \eqref{imd} and the similar bound for $I^{\al}_{M}$ are identical to the one for \eqref{nmd}.
\end{proof}


\section{LWP for the normal form equation \eqref{nfe}}

Following the development of modulation estimates in Sobolev spaces, we use Proposition \ref{prop-nm} to show that the normal form equation is LWP in $H^s$, with $s\geq 0$, with the solution being unique in the natural solution space \eqref{xt}. For this purpose, we first prove favorable bounds for $N^{(k)}_0(v)$ ($k\geq 2$) and $N^{(k)}_1(v)$ ($k\geq 1$), which were introduced in \eqref{n02}, \eqref{n0k}, \eqref{n11}, \eqref{n12}, and \eqref{n1k}. The key observation is that both $N^{(k)}_0$ and $N^{(k)}_1$ are linear combinations of multilinear operators, which are, in turn, compositions of specific bilinear operators of the type defined by \eqref{nam} and \eqref{iam}. For example,
\begin{align}
N^{(2)}_0(v)(t,x)=-i\,I^{0}_{> N}(v(t),v(t))(t,x),\label{n02-2}
\end{align}
\begin{equation}
\aligned
N^{(3)}_0(v)(t,x)=&\,I^{0}_{> N}(I^{-\tm_1}_{> 5^3|\tm_1|^{1-\delta}}(v(t),v(t))(t), v(t))(t,x)\\
&\qquad+I^{0}_{> N}(v(t), I^{-\tm_1}_{> 5^3|\tm_1|^{1-\delta}}(v(t),v(t))(t))(t,x),
\endaligned
\end{equation}
\begin{align}
N^{(1)}_1(v)(t,x)=\,N^{0}_{\leq N}(v(t),v(t))(t,x),\label{n11-v2}
\end{align}
\begin{equation}
\aligned
N^{(2)}_1(v)(t,x)=&\,i\bigg(I^{0}_{> N}(N^{-\tm_1}_{\leq 5^3|\tm_1|^{1-\delta}}(v(t),v(t))(t), v(t))(t,x)\\
&\qquad+I^{0}_{> N}(v(t), N^{-\tm_1}_{\leq 5^3|\tm_1|^{1-\delta}}(v(t),v(t))(t))(t,x)\bigg).
\endaligned
\label{n21-v2}
\end{equation}

In fact, the generic term featured in the formulas \eqref{n0k} and \eqref{n1k} for $N^{(k)}_0$ and $N^{(k)}_1$ originates from the compositions (in reverse order) of 
\begin{equation*}
I^{0}_{> N},\quad I^{-\tm_1}_{> 5^3|\tm_1|^{1-\delta}}, \quad I^{-\tm_2}_{> 7^3\max\{|\tm_1|, |\tm_2|\}^{1-\delta}},\ \ldots, \ I^{-\tm_{k-2}}_{> (2k-1)^3\max\{|\tm_1|, |\tm_{k-2}|\}^{1-\delta}},
\end{equation*}
and
\begin{equation*}\aligned
I^{0}_{> N},\quad I^{-\tm_1}_{> 5^3|\tm_1|^{1-\delta}}, \quad &I^{-\tm_2}_{> 7^3\max\{|\tm_1|, |\tm_2|\}^{1-\delta}},\ \ldots, \ I^{-\tm_{k-2}}_{> (2k-1)^3\max\{|\tm_1|, |\tm_{k-2}|\}^{1-\delta}},\\ 
&N^{-\tm_{k-1}}_{\leq (2k+1)^3\max\{|\tm_1|, |\tm_{k-1}|\}^{1-\delta}},
\endaligned
\end{equation*}
respectively. For a more formal treatment of the structures of $N^{(k)}_0$ and $N^{(k)}_1$ in the case of \eqref{kdv}, we refer the reader to Definition 3.13 and the comments following this definition in \cite{KOY18}. This can be easily adapted to our setting since the nonlinearity is quadratic instead of cubic. 

\begin{prop}
If $s\geq 0$, $N>1$, and $0<\delta<1$, then the following estimates hold true:
\begin{equation}
\|N^{(k)}_0(v)(t)\|_{H^s}\lesssim N^{-\frac{k-1}{2}+\frac{k-2}{2} \delta}\|v(t)\|^k_{H^s},
\label{n0k-2}
\end{equation}
\begin{equation}
\aligned
\|N^{(k)}_0(v)(t)-N^{(k)}_0(w)(t)\|_{H^s}\lesssim\ &N^{-\frac{k-1}{2}+\frac{k-2}{2} \delta}\left(\|v(t)\|^{k-1}_{H^s}+\|w(t)\|^{k-1}_{H^s}\right)\\
&\cdot\|v(t)-w(t)\|_{H^s},
\endaligned
\label{n0k-3}
\end{equation}
\begin{equation}
\|N^{(k)}_1(v)(t)\|_{H^s}\lesssim N^{-\frac{k-2}{2}+\frac{k-3}{2} \delta}\|v(t)\|^{k+1}_{H^s},
\label{n1k-2}
\end{equation}
\begin{equation}
\aligned
\|N^{(k)}_1(v)(t)-N^{(k)}_1(w)(t)\|_{H^s}\lesssim\ &N^{-\frac{k-2}{2}+\frac{k-3}{2} \delta}\left(\|v(t)\|^{k}_{H^s}+\|w(t)\|^{k}_{H^s}\right)\\
&\cdot\|v(t)-w(t)\|_{H^s},
\endaligned
\label{n1k-3}
\end{equation}
for all $k\geq 2$. In addition,
\begin{equation}
\|N^{(1)}_1(v)(t)\|_{H^s}\lesssim N^{1/2}\|v(t)\|^2_{H^s},
\label{n11-2}
\end{equation}
\begin{equation}
\aligned
\|N^{(1)}_1(v)(t)-N^{(1)}_1(w)(t)\|_{H^s}\lesssim\ &N^{1/2}\left(\|v(t)\|_{H^s}+\|w(t)\|_{H^s}\right)\\
&\cdot\|v(t)-w(t)\|_{H^s},
\endaligned
\label{n11-3}
\end{equation}
are valid.
\label{prop-nk-hs}
\end{prop}

\begin{proof}
First, we notice that \eqref{n0k-3}, \eqref{n1k-3}, and \eqref{n11-3} follows from \eqref{n0k-2}, \eqref{n1k-2}, and \eqref{n11-2}, respectively, in the same way we derived \eqref{nmd} from \eqref{nm}. Moreover, the bounds \eqref{n0k-2} for $k=2$ and \eqref{n11-2} are the direct consequences of \eqref{n02-2} and \eqref{im} and \eqref{n11-v2} and \eqref{nm}, respectively.  

Hence, in what follows, we focus on proving \eqref{n0k-2} for $k\geq 3$ and \eqref{n1k-2} for $k\geq 2$. Due to the generic nature of the terms in the expressions for $N^{(k)}_0$ and $N^{(k)}_1$, which was previously discussed, we analyze only one of them. We choose to work with
\begin{equation*}
T(v)=I^{0}_{> N}\big(v, I^{-\tm_1}_{> 5^3|\tm_1|^{1-\delta}}\big(v, \ldots I^{-\tm_{k-2}}_{> (2k-1)^3\max\{|\tm_1|, |\tm_{k-2}|\}^{1-\delta}}\big(v,v\underbrace{\big)\big)\ldots\big)}_{k-1 \, \text{times}},
\end{equation*}
and\footnote{For $N^{(2)}_1$, the formula of $S(v)$ needs a slight adjustment (see \eqref{n21-v2}).}
\begin{equation*}
\aligned
S(v)=I^{0}_{> N}\big(v, I^{-\tm_1}_{> 5^3|\tm_1|^{1-\delta}}\big(v, \ldots &I^{-\tm_{k-2}}_{> (2k-1)^3\max\{|\tm_1|, |\tm_{k-2}|\}^{1-\delta}}\big(v,\\ 
&N^{-\tm_{k-1}}_{\leq (2k+1)^3\max\{|\tm_1|, |\tm_{k-1}|\}^{1-\delta}}\big(v,v\underbrace{\big)\big)\ldots\big)}_{k \, \text{times}},
\endaligned
\end{equation*}
where we dropped the dependence of $v$ in terms of $t$ to simplify the notation. We perform a dyadic analysis in terms of the values taken by $\tm_1$, $\ldots,\, \tm_{k-1}$ and we rely on the notation 
\begin{equation*}
N_l\in 2^{\mathbb{Z}}, \quad |\tm_l|\sim N_l, \quad M_l=\max\{N_1, N_l\}, \qquad  (\forall)\,1\leq l\leq k,
\end{equation*}
which imposes the constraints  
\begin{equation*}
N_1\geq N,\ N_2\geq 5^3M_1^{1-\delta}, \ldots,\ N_{k-1}\geq (2k-1)^3M_{k-2}^{1-\delta}
\end{equation*}
for $T$ and
\begin{equation*}
N_1\geq N,\ N_2\geq 5^3M_1^{1-\delta}, \ldots,\ N_{k-1}\geq (2k-1)^3M_{k-2}^{1-\delta},\ N_{k}\leq (2k+1)^3M_{k-1}^{1-\delta}
\end{equation*}
for $S$.

By consecutively applying the results of Proposition \ref{prop-nm} for the operators
\[
I^0_{N_1},\ I^{-\tm_1}_{N_2}, \ldots,\  I^{-\tm_{k-2}}_{N_{k-1}},
\] 
and taking advantage of $M_l\geq N_1$ for all $l$, we derive
\begin{equation}
\aligned
\|T(v)\|_{H^s}&\lesssim \left(\sum_{N_1\geq N}\sum_{N_2\geq 5^3M_1^{1-\delta}}\ldots\sum_{N_{k-1}\geq (2k-1)^3M_{k-2}^{1-\delta}}\ \prod_{1\leq l\leq k-1}N_l^{-1/2}\right) \|v\|^k_{H^s}\\
&\lesssim \prod_{2\leq l\leq k-1}(2l+1)^{-3/2}\sum_{N_1\geq N}N_1^{-\frac{1}{2}-(k-2)\frac{1-\delta}{2}}\ \|v\|^k_{H^s}\\
&\lesssim \prod_{2\leq l\leq k-1}(2l+1)^{-3/2}\ N^{-\frac{k-1}{2}+\frac{k-2}{2}\delta}\ \|v\|^k_{H^s}.
\endaligned
\label{Tv}
\end{equation}
As explained in Lemma 3.15 of \cite{KOY18}, this is an estimate which leads to \eqref{n0k-2} since, for a fixed value of $s$,
\[
2^{ks}|\gT(k)|\prod_{2\leq l\leq k-1}(2l+1)^{-3/2}\lesssim 1
\]
uniformly in $k$.

Next, we turn to proving a favorable bound for $S(v)$, for which we use in succession Proposition \ref{prop-nm} for the operators
\[
I^0_{N_1},\ I^{-\tm_1}_{N_2}, \ldots,\  I^{-\tm_{k-2}}_{N_{k-1}},\ N^{-\tm_{k-1}}_{N_{k}}.
\] 
When compared to the analysis for $T(v)$, we need to consider here separately the cases when
\[
M_{k-1}=N_{k-1}\qquad \text{or}\qquad M_{k-1}=N_{1}.
\]
Thus, we are able to infer that
\begin{equation*}
\aligned
&\|S(v)\|_{H^s}\\
&\ \lesssim \left(\sum_{N_1\geq N}\ldots\sum_{N_{k-1}\geq (2k-1)^3M_{k-2}^{1-\delta}}\sum_{N_{k}\leq (2k+1)^3M_{k-1}^{1-\delta}}\prod_{1\leq l\leq k-1}N_l^{-1/2}\cdot N_k^{1/2}\right) \|v\|^{k+1}_{H^s}\\
&\ \lesssim \left(\sum_{N_1\geq N}\ldots\sum_{N_{k-1}\geq (2k-1)^3M_{k-2}^{1-\delta}}\sum_{N_{k}\leq (2k+1)^3N_{k-1}^{1-\delta}}\prod_{1\leq l\leq k-1}N_l^{-1/2}\cdot N_k^{1/2}\right) \|v\|^{k+1}_{H^s}\\
&\ \ \ + \left(\sum_{N_1\geq N}\ldots\sum_{N_{k-1}\geq (2k-1)^3M_{k-2}^{1-\delta}}\sum_{N_{k}\leq (2k+1)^3N_{1}^{1-\delta}}\prod_{1\leq l\leq k-1}N_l^{-1/2}\cdot N_k^{1/2}\right) \|v\|^{k+1}_{H^s}
\endaligned
\end{equation*}
\begin{equation*}
\aligned
&\ \lesssim \prod_{2\leq l\leq k-2}(2l+1)^{-3/2}\cdot (2k+1)^{3/2}\\
&\qquad\qquad\quad\cdot\sum_{N_1\geq N}N_1^{-\frac{1}{2}-(k-3)\frac{1-\delta}{2}}\left(\sum_{N_{k-1}\geq (2k-1)^3N_1^{1-\delta}}N_{k-1}^{-\frac{1}{2}+\frac{1-\delta}{2}}\right)\ \|v\|^{k+1}_{H^s}\\
&\ \ \ + \prod_{2\leq l\leq k-1}(2l+1)^{-3/2}\cdot (2k+1)^{3/2}\sum_{N_1\geq N}N_1^{-\frac{1}{2}-(k-2)\frac{1-\delta}{2}+\frac{1-\delta}{2}}\ \|v\|^{k+1}_{H^s}\\
&\ \lesssim \prod_{2\leq l\leq k-2}(2l+1)^{-3/2}\cdot (2k+1)^{3/2}\bigg((2k-1)^{-3\delta/2} N^{-\frac{1}{2}-(k-3)\frac{1-\delta}{2}-\frac{(1-\delta)\delta}{2}}\\
&\qquad\qquad\quad\qquad\qquad\qquad\qquad\qquad\quad+(2k-1)^{-3/2} N^{-\frac{1}{2}-(k-3)\frac{1-\delta}{2}}\bigg)\ \|v\|^{k+1}_{H^s}\\
&\ \lesssim \prod_{2\leq l\leq k-2}(2l+1)^{-3/2}\cdot (2k-1)^{-3\delta/2}(2k+1)^{3/2} N^{-\frac{k-2}{2}+\frac{k-3}{2}\delta}\ \|v\|^{k+1}_{H^s}.
\endaligned
\end{equation*}
This bound implies \eqref{n1k-2} by reasoning in the same way \eqref{n0k-2} was deduced from \eqref{Tv}. 
\end{proof}

Now, we have all the prerequisites to set up a contraction argument which proves that  \eqref{nfe} is LWP in $H^s$ when $s\geq 0$.

\begin{theorem}
If $s\geq 0$ and $r\geq 1$, then, for any $\|u_0\|_{H^s}\leq r$, there exist $T=T(r)>0$ and $N=N(r)>1$ such that the normal form equation \eqref{nfe} admits a unique solution $v\in X=C([-T,T]; H^s(\R))$ and the data-to-solution map
\begin{equation*}
u_0\in \{z;\ \|z\|_{H^s}\leq r\}\mapsto v\in X
\end{equation*}
is Lipschitz continuous.
\label{lwp-nfe}
\end{theorem}

\begin{proof}
We proceed in the standard way and denote the right-hand side of \eqref{nfe} by $L_{u_0}= L_{u_0}(v)$ , for a fixed $u_0\in H^s$. The goal is to show that $L_{u_0}$ is a contraction map on a closed ball of $X$. In fact, if $\|u_0\|_{H^s}\leq r$, we prove that we can take this ball to be $B(0, 2r)=\{v; \|v\|_X\leq 2r\}$ by choosing $T$ and $N$ appropriately. 

For this purpose, we let $C>0$ be an absolute constant which is valid for all the estimates \eqref{n0k-2}-\eqref{n11-3} (i.e., one can replace $\lesssim$ by $\leq C\cdot$). A direct application of \eqref{n0k-2}-\eqref{n11-3} in the context of \eqref{nfe} yields
\begin{equation*}
\aligned
\|L_{u_0}(v)\|_X\leq \|u_0\|_{H^s}+\, &C\sum_{k\geq 2} N^{-\frac{k-1}{2}+\frac{k-2}{2} \delta}\left(\|v\|^k_X+\|u_0\|^k_{H^s}\right)\\ +\, &C\,T\left(N^{1/2}\|v\|^2_X+\sum_{k\geq 2} N^{-\frac{k-2}{2}+\frac{k-3}{2} \delta}\|v\|^{k+1}_X\right).
\endaligned
\end{equation*}
If we take $v\in B(0,2r)$ and $N$ to satisfy 
\[
2rN^{-\frac{1-\delta}{2}}<\frac{1}{2},
\]
then the above two power series are convergent and we infer that
\begin{equation*}
\|L_{u_0}(v)\|_X\leq r+C\left(10r^2N^{-1/2}+4r^2TN^{1/2}+16r^3TN^{-\delta/2}\right).
\end{equation*}
By further enforcing
\[
\frac{15C}{2}<N^{\delta/2} \qquad \text{and}\qquad T\leq \frac{1}{12CrN^{1/2}},
\]
we deduce that $L_{u_0}: B(0,2r)\to B(0,2r)$ is well-defined.

Next, arguing in similar fashion, we derive 
\begin{equation*}
\aligned
\|&L_{u_0}(v)-L_{u_0}(w)\|_X\\
&\quad \leq C\left(8rN^{-1/2}+4rTN^{1/2}+16r^2TN^{-\delta/2}\right) \|v-w\|_X, \quad (\forall)\, v, w\in B(0,2r).
\endaligned
\end{equation*}
If we also include the restrictions
\[
(24Cr)^2<N \qquad \text{and}\qquad (4r)^{\frac{2}{1+\delta}}<N,
\] 
the map $L_{u_0}$ becomes a contraction on $B(0,2r)$. Furthermore, without any extra adjustments to the values of $T$ and $N$, one obtains in the same manner that 
\begin{equation*}
\aligned
\|L_{u_0}(v)-L_{\widetilde{u}_0}(\widetilde{v})\|_X \leq \|u_0-\widetilde{u}_0\|+&\,\widetilde{C}\, \|v-\widetilde{v}\|_X,\\
&(\forall)\, \|u_0\|_{H^s},  \|\widetilde{u}_0\|_{H^s}\leq r, \ v, \widetilde{v}\in B(0,2r),
\endaligned
\end{equation*}
for some fixed $0<\widetilde{C}<1$. This proves the assertion about the data-to-solution map.
\end{proof}


\section{Localized modulation estimates in $\mathcal{F}L^\infty$ and $H^s$}\label{mod-li-hs}
   
In this section, we develop modulation estimates which allow us to justify certain steps in the iteration scheme implemented in Section \ref{sect-iter}. For this purpose, we introduce the bilinear operators $N^{\al}_{j, \leq M}=N^{\al}_{j, \leq M}(v_1,v_2)$ and $I^{\al}_{j, >M}=I^{\al}_{j, >M}(v_1,v_2)$ by
\begin{equation}
\widehat{N^{\al}_{j, \leq M}(v_1,v_2)}(t,\xi)= \int\limits_{\stackrel{\xi=\xi_1+\xi_2}{|\Phi(\xi, \xi_1,\xi_2)-\al |\leq M}}|\xi|^s|\xi_j|^{1-s}\, e^{it\Phi(\xi, \xi_1,\xi_2)}\,\widehat{v_1}(\xi_1)\,\widehat{v_2}(\xi_2)
\label{najm}
\end{equation}
and
\begin{equation}
\widehat{I^{\al}_{j, > M}(v_1,v_2)}(t,\xi)= \int\limits_{\stackrel{\xi=\xi_1+\xi_2}{|\Phi(\xi, \xi_1,\xi_2)-\al |>  M}}|\xi|^s|\xi_j|^{1-s}\, \frac{e^{it\Phi(\xi, \xi_1,\xi_2)}}{\Phi(\xi, \xi_1,\xi_2)-\al }\,\widehat{v_1}(\xi_1)\,\widehat{v_2}(\xi_2),
\label{iajm}
\end{equation}
with obvious definitions for $N^{\al}_{j, M}=N^{\al}_{j, M}(v_1,v_2)$ and $I^{\al}_{j, M}=I^{\al}_{j, M}(v_1,v_2)$ in the spirit of the previous section. As before, $M\geq 1$ and $\al$ are real parameters, $\Phi$ is the modulation function, and $j=1$ or $2$.

\begin{prop}
If $0\leq s\leq \min\{1,\sigma\}$, then the following estimates are valid:
\begin{equation}
\|N^{\al}_{j, \leq  M}(v_1,v_2)(t)\|_{\mathcal{F}L^\infty}\lesssim M^{1/2}\|v_j\|_{\mathcal{F}L^\infty}\|v_{3-j}\|_{H^\sigma},\label{nmj}
\end{equation}
\begin{equation}
\|I^{\al}_{j, > M}(v_1,v_2)(t)\|_{\mathcal{F}L^\infty}\lesssim M^{-1/2}\|v_j\|_{\mathcal{F}L^\infty}\|v_{3-j}\|_{H^\sigma}.
\label{imj}
\end{equation}
$N^{\al}_{j, M}$ and $I^{\al}_{j, M}$ satisfy bounds identical to the ones for $N^{\al}_{j, \leq M}$ and $I^{\al}_{j, >M}$, respectively.
\label{prop-nmj}
\end{prop}

\begin{proof}
First, we notice that we only need to prove \eqref{nmj}, since the other estimates in the proposition are derived from it in the same way \eqref{im} and the corresponding bounds for $N^\alpha_M$ and $I^\alpha_M$ in Proposition \ref{prop-nm} are derived from \eqref{nm}. We assume, without loss of generality, that $j=1$ and, in this case, it is easy to see by duality that \eqref{nmj} is the consequence of 
\begin{equation}
\aligned
\sup_{\xi\in\R}\ \int\limits_{\xi=\xi_1+\xi_2}\bc_{\{|\Phi(\xi, \xi_1,\xi_2)-\al |\leq M\}}\,\frac{|\xi|^s|\xi_1|^{1-s}}{\langle \xi_2\rangle^{\sigma}} \,w(\xi_2)
\lesssim M^{1/2} \|w\|_{L^2}, 
\endaligned
\label{nmjw}
\end{equation}
with $w\in L^2$ being a non-negative function. We demonstrate this inequality by analyzing the following complementary scenarios: 
\[
\aligned |\xi_1|\lesssim 1, \qquad \big(\max\{|\xi_2|,1\}&\ll |\xi_1| \ \text{or}\ 1\ll |\xi_1|\ll |\xi_2|\big), \qquad |\xi|\lesssim 1\ll |\xi_1|\sim |\xi_2|,\\
&\text{and} \qquad 1\ll |\xi|\lesssim |\xi_1|\sim |\xi_2|.
\endaligned
\]

\underline{Case: $|\xi_1|\lesssim 1$.} The triangle inequality implies $|\xi|\lesssim \langle \xi_2\rangle$ and, consequently, we have
\begin{equation*}
\frac{|\xi|^s|\xi_1|^{1-s}}{\langle \xi_2\rangle^{\sigma}}\lesssim 1.
\end{equation*}
Then, by using the Cauchy-Schwarz inequality, we derive that
\begin{equation*}
\aligned
\text{LHS of \eqref{nmjw}}\lesssim \sup_{\xi\in \R}\left(\ \int\limits_{|\xi_1|\lesssim 1} w^2(\xi-\xi_1)d\xi_1\right)^{1/2}\lesssim \|w\|_{L^2},
\endaligned
\end{equation*}
which proves the claim.

\underline{Case: $\big(\max\{|\xi_2|,1\}\ll |\xi_1|$ or $1\ll |\xi_1|\ll |\xi_2|\big)$.} In this scenario, the triangle inequality yields 
\begin{equation}
\max\left\{1, \min\{|\xi_1|, |\xi_2|\}\right\}\ll \max\{|\xi_1|, |\xi_2|\}\sim |\xi|, 
\label{min-max}
\end{equation}
and, hence, we obtain that
\begin{equation*}
\frac{|\xi|^s|\xi_1|^{1-s}}{\langle \xi_2\rangle^{\sigma}}\lesssim |\xi|.
\end{equation*}
Following this, we proceed like in the \underline{Case $\max\{|\xi_2|,1\}\ll |\xi|$} of the argument for \eqref{nm}. We fix $\xi$, take $G(\xi_1)=\Phi(\xi_1,\xi-\xi_1, \xi)$, and, based on \eqref{g'} and \eqref{min-max}, infer that
\begin{equation*}
|G'(\xi_1)|=\big|\xi_1^2-\xi_2^2\big|(5(\xi_1^2+\xi_2^2)+3\beta)\sim \max\big\{\xi^4_1, \xi^4_2\big\}\sim \xi^4.
\end{equation*}
As argued before, we deduce that the domain in which $\xi_1$ varies has size $\lesssim M/\xi^4$. By invoking the Cauchy-Schwarz inequality, we finally derive that
\begin{equation*}
\aligned
\text{LHS of \eqref{nmjw}}&\lesssim \sup_{\xi\in\R}\left(\ \int\limits_{\xi=\xi_1+\xi_2}\bc_{\{|\Phi(\xi, \xi_1,\xi_2)-\al |\leq M\}}\,\frac{|\xi|^{2s}|\xi_1|^{2-2s}}{\langle \xi_2\rangle^{2\sigma}}\right)^{1/2}\   \|w\|_{L^2}\\
&\lesssim \sup_{|\xi|\gg 1}\left(\xi^2 \int_\R\bc_{\{|G(\xi_1)-\al |\leq M\}\cap \{|G'(\xi_1)|\sim \xi^4\}}\,d\xi_1\right)^{1/2}\   \|w\|_{L^2}\\
&\lesssim \sup_{|\xi|\gg 1}\frac{M^{1/2}}{|\xi|} \  \|w\|_{L^2}\sim M^{1/2}\|w\|_{L^2},
\endaligned
\end{equation*}
which demonstrates the assertion.

\underline{Case: $|\xi|\lesssim 1\ll |\xi_1|\sim |\xi_2|$.} For this case, we have that
\begin{equation*}
\frac{|\xi|^s|\xi_1|^{1-s}}{\langle \xi_2\rangle^{\sigma}}\lesssim |\xi_2|^{1-s-\sigma}.
\end{equation*}
Next, we argue like in the \underline{Case: $1\ll |\xi|\lesssim |\xi_1|\sim |\xi_2|$} of the proof of \eqref{nm}. We fix $\xi_2$, write $H(\xi_1)=\Phi(\xi_1,\xi_2, \xi_1+\xi_2)$, and, using \eqref{h'}, infer that
\begin{equation*}
|H'(\xi_1)|=|\xi_2||\xi+\xi_1|(5(\xi_1^2+\xi^2)+3\beta)\sim |\xi_2||\xi_1|^3\sim \xi_2^4.
\end{equation*}
Consequently, the set in which $\xi_1$ takes values has size $\lesssim M/\xi_2^4$. By applying the Cauchy-Schwarz inequality as in the previous case, we deduce that
\begin{equation*}
\aligned
\text{LHS of \eqref{nmjw}}&\lesssim \sup_{|\xi_2|\gg 1}\left(|\xi_2|^{2-2s-2\sigma} \int_\R\bc_{\{|H(\xi_1)-\al |\leq M\}\cap \{|H'(\xi_1)|\sim \xi_2^4\}}\,d\xi_1\right)^{1/2}\   \|w\|_{L^2}\\
&\lesssim \sup_{|\xi_2|\gg 1}\frac{M^{1/2}}{|\xi_2|^{1+s+\sigma}} \  \|w\|_{L^2}\sim M^{1/2}\|w\|_{L^2},
\endaligned
\end{equation*}
which yields the claim.

\underline{Case: $1\ll |\xi|\lesssim |\xi_1|\sim |\xi_2|$.} In this instance, we have that 
\begin{equation*}
\frac{|\xi|^s|\xi_1|^{1-s}}{\langle \xi_2\rangle^{\sigma}}\lesssim |\xi_1|^{1-\sigma}\sim |\xi_2|^{1-\sigma}
\end{equation*}
and we proceed in identical fashion with the way we argued for \eqref{nm} under the same assumptions. We infer that
\begin{equation*}
\max\{|\xi+\xi_1|, |\xi+\xi_2|\}\gtrsim |\xi|\gg 1
\end{equation*}
and choose to work with $\max\{|\xi+\xi_1|, |\xi+\xi_2|\}=|\xi+\xi_1|$. We claim that when $\max\{|\xi+\xi_1|, |\xi+\xi_2|\}=|\xi+\xi_2|$ the analysis is similar, by interchanging the roles of $\xi_1$ and $\xi_2$. Hence, yet again, we fix $\xi_2$, consider $H=H(\xi_1)$, and, using \eqref{h'}, derive that $|H'(\xi_1)|\gtrsim |\xi_2|^3$. Now, the outcome is that the domain in which $\xi_1$ varies has size $\lesssim M/|\xi_2|^3$.
Therefore, much in the same spirit with the previous case, we obtain
\begin{equation*}
\aligned
\text{LHS of \eqref{nmjw}}&\lesssim \sup_{|\xi_2|\gg 1}\left(|\xi_2|^{2-2\sigma} \int_\R\bc_{\{|H(\xi_1)-\al |\leq M\}\cap \{|H'(\xi_1)|\sim |\xi_2|^3\}}\,d\xi_1\right)^{1/2}\   \|w\|_{L^2}\\
&\lesssim \sup_{|\xi_2|\gg 1}\frac{M^{1/2}}{|\xi_2|^{1/2+\sigma}} \  \|w\|_{L^2}\sim M^{1/2}\|w\|_{L^2},
\endaligned
\end{equation*}
which finishes the proof of \eqref{nmj}.
\end{proof}


\section{Conclusion of the argument for Theorem \ref{main-th}}

In this final section, we complete the argument for our main result by justifying the heuristics in Section \ref{just}, which connect the Cauchy problem \eqref{main} to the normal form equation \eqref{nfe}. This process is very similar to the one used in Section 4.2 (and the related Section 4.1) of \cite{KOY18} to justify identical claims for \eqref{kdv}. This is why we present here only the key elements of this procedure and ask the reader to fill in the details by following the line of reasoning in \cite{KOY18}.

First, in order to justify the application of Leibniz's rule with respect to $\tau$, it is enough to show that $t\mapsto \widehat{v}(t,\xi)$ is a $C^1$-function for each fixed $\xi$. On the basis of \eqref{ir} and \eqref{v1-2}, we infer that
\begin{equation*}
\partial_t\widehat{v}(t,\xi)=-2\pi i\,\xi \,e^{-it(\xi^5+\beta\xi^3)}\,\widehat{u^2}(t,\xi).
\end{equation*}
However, if $u\in C(\R; H^s(\R))$ with $s\geq 0$, then $u^2\in C(\R; L^1(\R))$ and the desired claim follows according to the Riemann-Lebesgue lemma.

Secondly, in order to explain rigorously why we can commute differentiation with integration in the derivation of \eqref{nfe-k} and why this equation is a true approximation for \eqref{nfe}, we argue that it is enough to demonstrate the following result\footnote{This is the counterpart of Lemma 4.8 in \cite{KOY18}, used to justify the same heuristics in the case of \eqref{kdv}.}.

\begin{prop}
If $0\leq s\leq 1$, $N>1$, and $0<\delta<1$, then the following estimates are valid for all $k\geq 2$:
\begin{equation}
\left|\widehat{N^{(k)}(v)}(t,\xi)\right|\lesssim |\xi|^{1-s}N^{-\frac{k-1}{2}+\frac{k-2}{2} \delta}\|v(t)\|^{k+1}_{H^s},
\label{nk-fli}
\end{equation}
\begin{equation}
\left|\widehat{N^{(k)}_2(v)}(t,\xi)\right|\lesssim |\xi|^{1-s} N^{-\frac{k-1}{2}+\frac{k-2}{2} \delta}\|v(t)\|^{k+1}_{H^s}.
\label{nk-2}
\end{equation}
\label{prop-nk}
\end{prop}

In addition to the modulation estimates in Proposition \ref{prop-nm} and Proposition \ref{prop-nmj}, we need one more ingredient for the proof of the above proposition, which is the following bilinear estimate.

\begin{prop}
If $s\geq 0$, then
\begin{equation}
\sup_{\xi\in \R}\left|\ \int\limits_{\xi=\xi_1+\xi_2}|\xi|^s\,\widehat{v_1}(\xi_1)\,\widehat{v_2}(\xi_2)\right|\lesssim \|v_1\|_{H^s}\|v_2\|_{H^s}
\label{hs-hs}
\end{equation}
holds true.
\end{prop}
\begin{proof}
The argument is straightforward, as $s\geq 0$ implies 
\[
|\xi|^s\lesssim |\xi_1|^s+|\xi_2|^s
\]
and, subsequently, an application of the Cauchy-Schwarz inequality yields
\[
\aligned
\text{LHS of \eqref{hs-hs}}&\lesssim \left\||\xi_1|^s\,\widehat{v_1}(\xi_1)\right\|_{L_{\xi_1}^2}\|\widehat{v_2}\|_{L^2}+ \|\widehat{v_1}\|_{L^2}\left\||\xi_2|^s\,\widehat{v_2}(\xi_2)\right\|_{L_{\xi_2}^2}\lesssim \|v_1\|_{H^s}\|v_2\|_{H^s}.
\endaligned
\]
\end{proof}

Next, we finish this section and the entire paper by providing:

\begin{proof}[Sketch of proof for Proposition \ref{prop-nk}]
The argument is similar to the one for Proposition \ref{prop-nk-hs}. As was the case there, the key observation is that, according to  \eqref{nk} and \eqref{n2k}$, \widehat{N^{(k)}}$ and $\widehat{N^{(k)}_2}$ are linear combinations with constant coefficients of integral terms which can be written conveniently using the operators defined by \eqref{iam} and \eqref{iajm}.

For example, one integral term featured in the summation formula for $\widehat{N^{(3)}}$ is
\begin{equation*}\aligned
\mathcal{I}=-&\int\limits_{\stackrel{\xi=\xi_1+\xi_2}{\stackrel{\xi_1=\xi_3+\xi_4}{\xi_2=\xi_5+\xi_6}}} \bc_{C_0\cap C_1^c}\,i\xi\,i\xi_1\,i\xi_2\,\frac{e^{i\tau\tm_3}}{i\tm_1i\tm_2}\prod_{3\leq a\leq 6}\,\widehat{v}(\tau,\xi_a)
=-i\sgn(\xi)|\xi|^{1-s}\\
&\cdot\int\limits_{\stackrel{\xi=\xi_1+\xi_2}{|\tm_1|>N}} \bigg\{|\xi|^s|\xi_2|^{1-s}\frac{e^{i\tau\tm_1}}{\tm_1}\int\limits_{\stackrel{\xi_1=\xi_3+\xi_4}{|\mu_2+\tm_1|>5^3|\tm_1|^{1-\delta}}} \xi_1\frac{e^{i\tau\mu_2}}{\mu_2+\tm_1}\widehat{v}(\tau,\xi_3)\widehat{v}(\tau,\xi_4)\\
&\qquad\qquad\ \cdot\int\limits_{\xi_2=\xi_5+\xi_6} \sgn(\xi_2)|\xi_2|^{s}e^{i\tau\mu_3}\widehat{v}(\tau,\xi_5)\widehat{v}(\tau,\xi_6)\bigg\},
\endaligned
\end{equation*}
with
\begin{equation*}
\tm_1=\Phi(\xi,\xi_1, \xi_2), \quad\mu_2=\Phi(\xi_1,\xi_3, \xi_4), \quad\mu_3=\Phi(\xi_2,\xi_5, \xi_6).
\end{equation*}
If one defines the operator $H=H(v_1,v_2)$ by
\begin{equation*}
\widehat{H(v_1,v_2)}(t,\eta)=\int\limits_{\eta=\eta_1+\eta_2} \sgn(\eta)|\eta|^{s}e^{it\Phi(\eta,\eta_1,\eta_2)}\widehat{v_1}(\eta_1)\widehat{v_2}(\eta_2),
\end{equation*}
then, according to \eqref{iam} and \eqref{iajm}, we can write
\begin{equation*}\aligned
\mathcal{I}&=-i\sgn(\xi)|\xi|^{1-s}\int\limits_{\stackrel{\xi=\xi_1+\xi_2}{|\tm_1|>N}} \bigg\{|\xi|^s|\xi_2|^{1-s}\frac{e^{i\tau\tm_1}}{\tm_1} \mathcal{F}\left(I^{-\tm_1}_{>5^3|\tm_1|^{1-\delta}}(v(\tau),v(\tau))\right)(\tau,\xi_1)\\
&\qquad\qquad\qquad\qquad\qquad\quad\cdot \mathcal{F}\left(H(v(\tau),v(\tau))\right)(\tau,\xi_2)\bigg\}\\
&=-i\sgn(\xi)|\xi|^{1-s}\mathcal{F}\left(I^0_{2,>N}\left(I^{-\tm_1}_{>5^3|\tm_1|^{1-\delta}}(v(\tau),v(\tau))(\tau), H(v(\tau),v(\tau))(\tau)\right)\right)(\tau,\xi).
\endaligned
\end{equation*}
Furthermore, due to \eqref{hs-hs}, we deduce
\begin{equation}
\|H(v_1,v_2)(t)\|_{\mathcal{F}L^\infty}\lesssim \|v_1\|_{H^s}\|v_2\|_{H^s}.\label{H-fli}
\end{equation}
Therefore, if we also factor in \eqref{imj} and \eqref{im}, then we derive
\begin{equation*}
\aligned
|\mathcal{I}|&\lesssim |\xi|^{1-s}N^{-1/2}\|I^{-\tm_1}_{>5^3N^{1-\delta}}(v(\tau),v(\tau))(\tau)\|_{H^s}\|H(v(\tau),v(\tau))(\tau)\|_{\mathcal{F}L^\infty}\\
&\lesssim |\xi|^{1-s}N^{-1+\frac{\delta}{2}}\|v(\tau)\|^4_{H^s},
\endaligned
\end{equation*}
which matches exactly the estimate \eqref{nk-fli} for $k=3$.

In general, the typical integral term in the formula for $N^{(k)}$ can be written as a composition of operators of the types
\[
I^{-\tm_{l}}_{> (2l+3)^3\max\{|\tm_1|, |\tm_{l}|\}^{1-\delta}}, \qquad I^{-\tm_{l}}_{> (2l+3)^3\max\{|\tm_1|, |\tm_{l}|\}^{1-\delta}}, \qquad H,
\]
with $1\leq l\leq k-2$ and $1\leq j\leq 2$. For $N^{(k)}_2$, the only addition would be a further modulation restriction imposed by $\bc_{C^c_{k-1}}$ on $H$, which doesn't bring any additional complications. All these considerations can be made rigorous by working with the concept of \emph{shortest path} between root nodes of a chronicle of $k$ generations. We refer the reader to Remark 3.6 and the formula (4.18) in \cite{KOY18} for precise details in the context of \eqref{kdv}. Following this, one takes advantage of \eqref{imj}, \eqref{im} and \eqref{H-fli} to obtain the estimate for the individual integral term. Finally, \eqref{nk-fli} and \eqref{nk-2} are derived in the same way \eqref{n0k-2} is deduced from \eqref{Tv}, which is a bound for one of the individual terms in the formula for $N^{(k)}_0$. 
\end{proof}

\section*{Acknowledgements}
The first author was supported in part by a grant from the Simons Foundation $\#\, 359727$. Both authors are grateful to Alex Iosevich for helpful discussions during the preparation of this manuscript.


\bibliographystyle{amsplain}
\bibliography{bousbib}

\end{document}